\documentclass[a4paper,11pt,twoside]{article}
\usepackage{amsmath,amsthm,amssymb,amscd}
\usepackage{ascmac}
\usepackage{enumerate}
\numberwithin{equation}{section}
\newtheorem{thm}{Theorem}[section]
\newtheorem{dfn}[thm]{Definition}
\newtheorem{exa}[thm]{Example}
\newtheorem{exas}[thm]{Examples}
\newtheorem{prop}[thm]{Proposition}
\newtheorem{cor}[thm]{Corollary}
\newtheorem{lem}[thm]{Lemma}
\newtheorem{rem}[thm]{Remark}

\def\End{\mathop{\mathrm{End}}\nolimits}

\newcommand{\mf}[1]{\mathfrak{#1}}

\newcommand{\bb}[1]{\mathbb{#1}}

\newcommand{\mrm}[1]{\mathrm{#1}}

\newcommand{\C}{\mathbb{C}}
\newcommand{\R}{\mathbb{R}}
\newcommand{\Z}{\mathbb{Z}}
\newcommand{\dif}{\mathrm{d}}
\newcommand{\pd}{\partial}
\newcommand{\Sc}{\mathrm{Sc}}
\newcommand{\Karea}{\mathrm{K\mathchar`-area}}
\newcommand{\ep}{\varepsilon}
\newcommand{\ua}{_{\alpha}}
\newcommand{\ub}{_{\beta}}
\newcommand{\uba}{_{\beta \alpha}}
\newcommand{\id}{\mathrm{id}}
\newcommand{\oshp}{^{\sharp}}

\newcommand{\norm}[1]{\Vert #1 \Vert}
\newcommand{\map}[3]{#1 \colon #2 \rightarrow #3}

\title{Invariance of Finiteness of K-area under Surgery}
\author{Yoshiyasu Fukumoto}
\date{}

\begin{document}
\maketitle

\begin{abstract}
K-area is an invariant for Riemannian manifolds introduced by Gromov
as an obstruction to the existence of positive scalar curvature.
However in general it is difficult to determine whether K-area is finite
or not in spite of its natural definition.
In this paper, we study how the invariant changes under surgery.

\end{abstract}
\section*{Introduction}
The notion of K-area was introduced by Gromov \cite{Gr96}.
It is an invariant for Riemannian manifolds with values in $(0,+\infty]$.
Roughly, $\Karea (M) $ measures how small $C^{0}$ curvature norms can be
achieved for "non-trivial" vector bundles over a Riemannian manifold $M$.
 Here a "non-trivial" vector bundle $E$ means a vector
bundle with non-zero Chern numbers. Finiteness of K-area has a deep
relationship with the existence of positive scalar curvature.
The following theorem was proved by Gromov using 
the relative index theorem \cite{Gr-La83}.
\begin{thm}\cite{Gr96} \label{thm:Sc and K-area}
Let $M$ be an even dimensional complete spin Riemannian manifold.
If the scalar curvature $\Sc$ of $M$ satisfies $\inf \Sc > \ep^{2}$,
then $\Karea (M) \leq c\ep^{-2}$ where $c$ is a constant depending on
the dimension of $M$.

In particular even dimensional spin manifold with $\Karea (M) = \infty$
does not admit complete Riemannian metrics of uniformly positive scalar
curvature.
\end{thm}
Even though both notions of scalar curvature and K-area require
Riemannian metrics, finiteness of K-area on a compact manifold depends
only on its homotopy type. Hence infiniteness of K-area is a homotopical
obstruction to the existence of positive scalar curvature on compact
spin manifolds.

In this paper we verify the following.
\begin{thm}\label{thm:surgery}
Let $M$ be an oriented even dimensional Riemannian manifold with
$p+q= \dim (M)$. Let $M\oshp$ be a manifold obtained by $p$-surgery
for $q \neq 2$.
If $\Karea (M) < \infty$, then $\Karea (M\oshp) < \infty$.
\end{thm}
As a special case:
\begin{cor} \label{cor:connected sum}
Let $M_{1}$ and $M_{2}$ be oriented even dimensional
Riemannian manifolds of the same dimension.
Let $M_{1} \sharp M_{2}$ denote the connected sum of 
$M_{1}$ and $M_{2}$.
If both $\Karea (M_{1})$ and $\Karea (M_{2})$ are finite,
then $\Karea (M_{1} \sharp M_{2})$ is also finite.
\end{cor}

On the other hand the converse is easy to verify. Of course, the
following lamma \ref{lem:conncetd sum trivial} follows
also from theorem \ref{thm:surgery}, but it can be verified without it.
\begin{lem}\label{lem:conncetd sum trivial}
Let $M_{1}$, $M_{2}$, and $M_{1} \sharp M_{2}$ be as above.
If either $M_{1}$ or $M_{2}$ has infinite $\Karea$,
then $\Karea(M_{1} \sharp M_{2}) =\infty$.
\end{lem}

In fact, if $M_{1}$ has infinite K-area then there exists a
"non-trivial" vector bundle $E$ over $M_{1}$ with small $C^{0}$
curvature norm. Then we can construct another vector bundle
over $M_{1} \sharp M_{2}$ by extending $E$ trivially onto
$M_{2}$.

We remark that the main theorem is analogous to the following.
\begin{prop} \cite{Gr-La80}
Let $M$ be a compact manifold which carries a Riemannian metric
of positive scalar curvature. Then any manifold obtained by surgeries
in codimension $\geq 3$
also carries a metric of positive scalar curvature.
\end{prop}

The proof of our main theorem is rather different.
The idea of the above proposition is that $S^{q-1} \times N$ admits
a Riemannian metric of positive scalar curvature for $q \geq 3$.
On the other hand, 
we use a property that the cartesian product of spheres at the connecting
region is simply connected. Any almost flat vector bundles over compact
simply connected manifolds are trivial, which will be used to compute
finiteness of K-area.

In \cite{Li10} M. Listing studies so called "homology classes of finite
K-area"
and remarks that the homology of finite K-area
in the dimensions lower than the largest one behave in the same way as
the ordinary homology when taking the connected sums.

In \cite{Ha11}
B. Hanke extends the concept of K-area by admitting Hilbert-$A$-module
bundles of small or vanishing curvature.
He defines the notion of infiniteness (and finiteness) of
K-area of $K$-homology classes $h \in K_{0}(M) \otimes \bb{Q}$
for closed smooth manifolds $M$.
It is shown that the K-area of the homological
fundamental classes of area-enlargeable 
manifolds in the sense of \cite{Gr-La83} are infinite.
Moreover he shows that oriented manifolds with fundamental
classes of infinite K-area are essential.
Manifolds are said to be essential if the classifying maps of
universal covers map the homological fundamental classes
to non-zero classes in the homology of the fundamental groups.

\section{Definition and a fundamental lemma}
Let $E\rightarrow M$ be a Hermitian vector bundle over a Riemannian 
manifold $M$, and let $A$ be a section of
$\bigwedge^{\ast}TM \otimes \End(E)$. Let us define
\begin{equation}
\norm{A}:= \sup_{\stackrel{\xi \in \bigwedge^{\ast}(TM)}{\norm{\xi}=1}}
\left|A(\xi) \right|_{op}
\end{equation}
where $\left|A(\xi) \right|_{op}$ denote the operator norm of
$A(\xi) \in \End(E)$.

Let $K^{\times}(M)$ denote the isomorphism classes of 
Hermitian vector bundles equipped with compatible connections
$E=(E,\nabla)$ over $M$, 
which satisfy the following conditions.
\\(i) $(E,\nabla)$ are isomorphic to the trivial bundles $\C^{r}$ 
equipped with flat connections outside compact subsets of $M$.
\\(ii) $(E,\nabla)$ have a non-zero Chern number.
i.e. there exists a (multivariable) polynomial $p$ such that
\begin{equation}
\int_{M}p(c_{1}(E),c_{2}(E), \cdots ) \neq 0
\end{equation}
where $c_{k}(E) \in H^{\ast}_{c}(M)$ are
the Chern classes of $E=(E,\nabla)$.

\begin{dfn}[\cite{Gr96}]
Let $M$ be an even dimensional Riemannian manifold and let
$ R= R^{E} = R^{E, \nabla} $ denote the curvature tensor of $(E,\nabla)$.
Then K-area of $M$ is defined by
\begin{eqnarray}
\Karea(M) := \sup_{(E,\nabla) \in K^{\times}(M)}
\frac{1}{\norm{R^{E,\nabla}}}
\end{eqnarray}
\end{dfn}
$\Karea(M) = \infty$ if and only if for any $\ep >0$, 
there exists a vector bundle $(E,\nabla) \in K^{\times}(M)$
with a small curvature $\norm{R} <\ep $.

The following fundamental lemma is useful. 
\begin{lem}\label{lem:hyper spherical}
Let $M$ and $M'$ be Riemannian manifolds and let $\map{f}{M}{M'}$
be a smooth Lipschitz map of non-zero degree which is proper or
constant outside a compact subset in $M$.
Then $\Karea(M) \geq c^{-2}\Karea(M')$ 
where $c$ is the Lipschitz constant of $f$.
\end{lem}
Lemma\ref{lem:hyper spherical} implies
that finiteness or infiniteness of $\Karea (M)$
is independent of the deformation of Riemannian metrics
on compact subsets in $M$.
In particular, the finiteness or infiniteness of $\Karea$
is a homotopy invariant of compact manifolds. This is stated in
\cite{Gr96} without proof. We give a proof for convenience.

\begin{proof}
Set $\Karea (M') = \frac{1}{a}$. If $\Karea (M') = \infty$, 
take $a=0$. For any $\ep >0$, there exists
$E=(E,\nabla) \in K^{\times}(M') $ with $\norm{R^{E}}<a+\ep $.
Let $p$ be a polynomial satisfying 
$\int_{M'} p(c_{1}(E),c_{2}(E), \cdots ) \neq 0 $.
Consider the vector bundle
$f^{\ast}E \rightarrow M$ equipped with the induced connection
$f^{\ast}\nabla$.
Since $f$ is proper or constant outside a compact subset,
$f^{\ast}E$ is isomorphic to a flat bundle
$\C ^{r}$ outside a compact subset.
Moreover, 
\begin{eqnarray}
\int_{M} p(c_{1}(f^{\ast}E),c_{2}(f^{\ast}E), \cdots ) = \deg (f)
\int_{M'}
p(c_{1}(E),c_{2}(E), \cdots ) \neq 0
\end{eqnarray}
Hence, $(f^{\ast}E, f^{\ast}\nabla) \in K^{\times}(M)$. On the other hand
\begin{eqnarray}
R^{f^{\ast}E}(u \wedge v) = R^{E}(f_{\ast}(u \wedge v))
\\
\norm{ R^{f^{\ast}E} } \leq \norm{f_{\ast}(u \wedge v)}
\norm{R^{E}}
\leq c^{2} (a+\ep)
\\
\Karea (M) \geq \frac{1}{\norm{ R^{f^{\ast}E} }} \geq
\frac{1}{c^{2}(a+\varepsilon)}
\end{eqnarray}
\\Therefore $\Karea (M) \geq c^{-2} \Karea(M')$
\end{proof}

Here, we give some examples of K-area.

\begin{exas}\label{exa:K-area}
\mbox{}\\
(1) Let $S^{2m}$ denote even dimensional spheres $\Karea(S^{2m})<\infty$,
which follows from theorem \ref{thm:Sc and K-area}.
\\
(2) If $M$ be an oriented even dimensional closed simply connected
manifold, then $\Karea(M)<\infty$.
Later in lemma \ref{lem:simply connected}, every vector bundle
$(E,\nabla)$ over a closed simply connected manifold with sufficiently
small curvature $\norm{R^{E,\nabla}}<\delta$ is topologically trivial,
which implies that all Chern numbers of $E$ are zero.
Hence $\Karea(M)<\frac{1}{\delta}$.
$\Karea (S^{n})<\infty$ can be verified also from this.
\\
(3) $\Karea(T^{2m})=\infty$ where $T^{2m}$ denote even dimensional 
tori. It follows from theorem \ref{thm:Sc and K-area} that $T^{2m}$
and hence $T^{2m-1}$ do not admit Riemannian metrics of positive
scalar curvature.
\end{exas}
\begin{proof}[proof of (3)]
Generally let $M=(M,g)$ be a Riemannian manifold equipped with a metric
$g$. Observe that $\Karea(M,c^{2}g)=c^{2}\Karea(M,g)$ by
the preceding lemma \ref{lem:hyper spherical}.

On the other hand let $\map{\pi}{\tilde{M}}{M}$ be a finite
covering space of $M$ which is trivial outside a compact subset.
Then $\Karea(\tilde{M})=\Karea(M)$.
In fact for $E=(E,\nabla) \in K^{\times}(\tilde{M})$, we can take
$\pi_{!}E \in K^{\times}$ whose fiber is 
\begin{eqnarray}
\pi_{!}E_{x}=\bigoplus_{\tilde{x}\in \pi^{-1}(x)} E_{\tilde{x}}
\end{eqnarray}
So we can verify that $\norm{R^{E}}\geq \norm{R^{\pi_{!}E}}$
and hence $\Karea(M)\geq \Karea(\tilde{M})$.
Conversely, $\map{\pi}{\tilde{M}}{M}$ satisfies the hypothesis of
the preceding lemma \ref{lem:hyper spherical} with Lipschitz constant
$c=1$ so $\Karea(\tilde{M})\geq \Karea(M)$.
Therefore $\Karea(\tilde{M})=\Karea(M)$.

Now consider an $2m$-dimensional tori equipped with flat metrics
$g_{0}$ which are induced by $T^{2m}=\R^{2m} \slash \Z^{2m}$.
There exist $2^{2m}$-fold coverings
$\map{\pi}{(T^{2m},4g_{0})}{(T^{2m},g_{0})}$.
Hence $\Karea(T^{2m},g_{0}) = \Karea(T^{2m},4g_{0}) =
4\Karea(T^{2m},g_{0})$, which implies $\Karea(T^{2m},g_{0}) = \infty$.
\end{proof}

\section{Surgery}
Let $M_{1}$ and $M_{2}$ be Riemannian manifolds and let 
$M_{1} \sharp M_{2}$ denote the connected sum of $M_{1}$ and $M_{2}$
equipped with a Riemannian metric which coincides with the original
metric of $M_{1} \sqcup M_{2}$ outside a compact neighborhood of
the connecting region.

\begin{exa}
Let $M$ be a 2m dimensional closed spin manifold.
Then $T^{2m} \sharp M$ does not admit a Riemannian metric
of positive scalar curvature.
In fact $\Karea (T^{2m})=\infty$ implies $\Karea (T^{2m}\sharp M) =
\infty$ and apply theorem \ref{thm:Sc and K-area}.
\end{exa}

\begin{proof}[proof of lemma \ref{lem:conncetd sum trivial}]
Suppose that $\Karea(M_{1}) = \infty$. Write $M_{1} \sharp M_{2}$ as
$(M_{1} \setminus D^{n}) \cup (M_{2} \setminus D^{n})$.
There exits a smooth map $\map{f}{(M_{1} \sharp M_{2})}{M_{1}}$
which satisfies the followings;
\\ \mbox{} \quad $f(M_{2} \setminus D^{n}) = \{x\}$
where $x$ is the center of $D^{n} \subset M_{1}$.
\\ \mbox{} \quad $f=\id$ outside a neighborhood of
$D^{n} \subset M_{1}$.
\\ \mbox{} \quad $\deg f =1$.
\\
Although $f$ does not necessarily satisfy the assumption of
lemma \ref{lem:hyper spherical},
we can see that $(f^{\ast}E, f^{\ast}\nabla) 
\in K^{\times}(M_{1} \sharp M_{2})$ if $(f,\nabla)\in K^{\times}(M_{1})$
just like as in the proof of lemma \ref{lem:hyper spherical}.
Hence it follows that
$\Karea(M_{1} \sharp M_{2}) \geq c^{-2}\Karea(M_{1})=\infty$
where $c$ is the Lipschitz constant of $f$.
\end{proof}

However, the converse of lemma \ref{lem:conncetd sum trivial} is not
trivial.
The following two lemmata are used to verify theorem \ref{thm:surgery}.

\begin{lem}\label{lem:simply connected}
Let $N$ be a compact simply connected Riemannian manifold and take
$E=(E, \nabla) \in K^{\times}(N)$.
For any $\ep>0$, there exist $\delta>0$ such that if
$\norm{R^{E}} < \delta$,
there exists a global orthonormal frame
$\bar{e}=\{\bar{e^{i}}\}_{i=1}^{r}$ for $E$ satisfying
$\norm{\omega}<\ep$ where $\omega$ is the connection 1-form of
$(E,\nabla)$ with respect to $\bar{e}$. 
\end{lem}

\begin{proof}
Fix a finite good open covering $\{ V\ua \}$ of $N$ equipped with geodesic
coordinates whose centers are $p\ua$. So
each finite intersection of $\{ V\ua \}$ is contractible unless it is empty.
Let $E$ be a Hermitian vector bundle with $\norm{R^{E}}<\delta$
for some $\delta >0$.
Fix an orthonormal basis $e_{\alpha_{0}}=\{e_{\alpha_{0}}^{i}\}_{i=1}^{r}$
for $E|_{p_{\alpha_{0}}}$. 
Let $e\ua=\{e\ua^{i}\}_{i=1}^{r}$ be an orthonormal basis for $E|_{p\ua}$
obtained by the parallel transportation of $e_{\alpha_{0}}$
along $\gamma _{0}^{\alpha} $, one of
the minimal geodesics connecting $p_{\alpha_{0}}$ and $p\ua$. 
Extend $e\ua$ on each $V\ua$ by the parallel transportation along the
geodesic
$t \mapsto \exp_{p\ua}(tv)$ where $v$ is an unit tangent vector at $p\ua$.

Let $\omega \ua$ be the connection 1-form with respect to $e\ua$ on
${V\ua}$. 

For $x \in V\ua$ let $\gamma \ua ^{x}$ be the (unique) geodesic
connecting $p\ua$ and $x$ and for a piece-wise smooth curve $\gamma$
let $T_\gamma$ be the  parallel transportation along $\gamma$.
Take $x \in U\ua$ and $X \in T_{x}M$. By the definition of $e\ua$,
$e\ua(\exp_{x}(tX))
=T_{\gamma \ua^{\exp_{x}(tX)}}T_{\gamma \ua^{x}}^{-1} e\ua (x)$.
Then,
\begin{eqnarray}
\nabla _{X}e\ua (x) &=& \lim_{t \rightarrow 0}\frac{1}{t}
\left( T_{\exp_{x}(tX)}^{-1} e\ua(\exp_{x}(tX))
- e\ua(x) \right)
\nonumber \\
&=& \lim_{t \rightarrow 0}\frac{1}{t}
\left( T_{\exp_{x}(tX)}^{-1}
T_{\gamma \ua^{\exp_{x}(tX)}}T_{\gamma \ua^{x}}^{-1}
-\id \right)e\ua(x)
\\
\norm{\nabla _{X}e\ua (x)} & \leq & \lim_{t \rightarrow 0}
\frac{1}{t} \int_{D_{t}} \norm{R}
\end{eqnarray}
where $D_{t}$ is a 2-dimensional disk whose boundary is
the closed curve
$\exp_{x}(tX)^{-1} \gamma \ua ^{\exp_{x}(tX)} (\gamma \ua^{x})^{-1}$.
Since $\mrm{area}(D_{t}) = O(t) \quad (t\rightarrow 0)$, we have 
\begin{eqnarray}
\norm{\nabla e\ua} \leq c_{1}\delta \label{eq:norm(nabla(e))}
\quad \mrm{i.e.} \quad
\norm{ \omega \ua } \leq c_{1}\delta
\end{eqnarray}
where $c_{1}$ is a constant depending on $\{ V_ua \}$. Keep in mind that
constants $c_{1}, c_{2}, \cdots, c_{6}$ which will appear below
are independent of the vector bundle $E$.

Let $\map{\psi \uba}{V\ua \cap V\ub}{U(r)}$ denote the
transition functions, i.e., $\psi \uba e\ua = e\ub$.
By the definition of $e\ua$, $\psi \uba (x)= T_{\gamma}$ 
where $\gamma = \gamma \ub ^{x} \gamma _{0}^{\beta} \left( \gamma \ua ^{x}
\gamma _{0}^{\alpha} \right)^{-1} $.
Since $N$ is simply connected, There exists a 2-dimensional
disk $D \subset N$ whose boundary is $\gamma$.
By the compactness, we can take $D$ so that $\mathrm{area}(D) < c_{2}$
where $c_{2}$ is a constant depending on $N$ and $\{V\ua \}$.
Then we have
\begin{eqnarray}\label{eq:norm(psi-id)}
\norm{\psi \uba - \id} \leq \int_{D} \norm{R}<c_{2}\delta
\end{eqnarray}

By $ \psi \uba e\ua = e\ub $, we have $\dif \psi \uba \otimes
e\ua + \psi \uba \nabla e\ua = \nabla e\ub$. Hence by
(\ref{eq:norm(nabla(e))})
\begin{eqnarray}
\norm{ \dif \psi \uba } < 2c_{1}\delta \label{eq:norm(dif(psi))}
\end{eqnarray}
Taking into account the estimate (\ref{eq:norm(psi-id)})
we can set
$\psi \uba = \exp (v \uba)$
for some $\map{v\uba}{V\ua \cap V\ub}{\mf{u}(r)}$
using $\map{\exp}{\mf{u}(r)}{U(r)}$
if $\delta >0$ is sufficiently small
since $\exp$ is a diffeomorphism from a neighbourhood of
$0 \in \mathfrak{u}(r)$
to a neighbourhood of ${\rm id} \in U(r)$.
Remark that (\ref{eq:norm(psi-id)}) and
(\ref{eq:norm(dif(psi))}) implies
\begin{eqnarray}\label{eq:norm(v)}
\norm{v\uba} < c'_{2}\delta \quad \mathrm{and} \quad \norm{\dif v\uba}
< 2c_{1}\delta
\end{eqnarray}

There exist open subsets $W\ua$ and compact subsets $K\ua$
such that $W\ua \subset K\ua \subset V\ua$ and $\bigcup \ua
W\ua = N$. Note that these are independent of $(E,\nabla)$.
Let $\{ \rho \ua, \rho \ub, 1- \rho \ua - \rho \ub \}$ be a partition of
unity associated to $\{ V\ua, V\ub, N\setminus(K\ua \cup K\ub) \}$.
($ \rho \ua + \rho \ub \equiv 1 $ on $K\ua \cup K\ub$.)
Construct an orthonormal frame $e_{(2)}$ on $K\ua \cup K\ub$ as follows;

\begin{eqnarray}\label{eq:e_(2)=}
e_{(2)}=
\left\{
\begin{array}{ll}
(\exp (\rho \ub v\uba))e\ua, &\quad \mathrm{on} K\ua
\\
(\exp (\rho \ua v_{\alpha \beta}))e\ub, &\quad \mathrm{on} K\ub

\end{array}
\right.
\end{eqnarray}
$e_{(2)}$ is well defined. In fact on $K\ua \cap K\ub$,
\begin{eqnarray}
\exp (\rho\ua v_{\alpha \beta})e\ub
&=& \exp(\rho\ua(-v\uba))\exp(v\uba)e\ua
\nonumber \\
&=& \exp((1-\rho\ub)(-v\uba)+v\uba)e\ua = \exp(\rho\ub v\uba)e\ua
\end{eqnarray}
There is a constant $c_{3}>0$ such that $|\dif \rho \ua|<c_{3},
|\dif \rho \ub|<c_{3}$.
Hence by (\ref{eq:norm(nabla(e))}), (\ref{eq:norm(v)}), and
(\ref{eq:e_(2)=}),
\begin{eqnarray}
\norm{\nabla e_{(2)}} &=& \norm{\dif (\exp(\rho\ub v\uba))\otimes e\ua +
\exp(\rho\ub v\uba)\nabla e\ua}
\nonumber \\
&\leq& |\dif \rho \ub| \norm{v\uba}+ \rho \ub \norm{\dif v\uba}
+\norm{\nabla e\ua} \qquad < c_{4} \delta
\end{eqnarray}
This means the connection 1-form $\omega _{(2)}$ associated to $e_{(2)}$
satisfies $\norm{\omega _{(2)}} < c_{4} \delta$.

Next, choose another open subset $V_{\gamma}$, set $V_{(2)}:= V\ua \cup
V\ub$, $K_{(2)}:= K\ua \cup K\ub$ 
and let $\psi_{\gamma(2)}\colon K_{\gamma} \cap K_{(2)} \rightarrow U(r)$ 
denote the transition function
i.e., $e_{\gamma} = \psi_{\gamma(2)} e_{(2)}$. 

Remark that $\id = \psi_{\gamma (2)} \exp(\rho \ub v\uba)
\psi _{\alpha \gamma}$ implies $\norm{\psi_{\gamma (2)}-\id} <
c_{5}\delta$ 
and $\norm {\dif \psi_{\gamma (2)}} < c_{5}\delta$. 
Therefore, we can write $ \psi_{\gamma (2)} 
= \exp (v_{\gamma (2)}) $ for some $v_{\gamma (2)}$ 
satisfying $ \norm{v_{\gamma (2)}} < c'_{5}\delta $ and $\norm{\dif
v_{\gamma (2)}} < c'_{5}\delta $.

We can employ the similar argument to construct an orthonormal frame
$e_{(3)}$ on $K_{\gamma}\cup K_{(2)}$ 
satisfying $\norm{\nabla e_{(3)}} < c_{6}\delta$.
Namely, let $\{ \rho_{\gamma}, \rho_{(2)}, 1-\rho_{\gamma}- \rho_{(2)} \}$
be a partition of
unity associated to $\{ V_{\gamma}, V_{(2)}, N\setminus(K_{\gamma} \cup
K_{(2)}) \}$, and define 
\begin{eqnarray}
e_{(3)}=
\left\{
\begin{array}{ll}
(\exp (\rho_{(2)} v_{(2)\gamma}))e_{\gamma}, &\quad \mathrm{on} K_{\gamma}
\\
(\exp (\rho_{\gamma} v_{\gamma (2)}))e_{(2)}, &\quad \mathrm{on} K_{(2)} 
\end{array}
\right.
\end{eqnarray}
It satisfies $\norm{\nabla e_{(3)}} < c_{6}\delta$.

Repeat the above argument to construct a global orthonormal frame
$\bar{e}$ for $E$ which satisfies $\norm{\nabla \bar{e}} < c\delta$.
It means $\norm{\omega}<c\delta$ where $\omega$ is the connection
1-form with respect to $\bar{e}$.
Though $c$ depends on $N$, it does not depend on $(E,\nabla)$.
\end{proof}

\begin{rem}
The proof of lemma \ref{lem:simply connected} also holds if 
$N$ is not connected but each connected component is
simply connected by applying the arguments on each connected component.
\end{rem}
\begin{lem}\label{lem:simply connected boundary}
Let $M$ be a Riemannian manifold with a simply connected boundary
$N=\pd M$, and let $E_{0}=(E_{0},\nabla_{0})$ be a Hermitian vector
bundle over $M$ equipped with a compatible connection.
Suppose that a neighborhood of $\pd M$ is equipped with a product metric
of $(-2,2] \times N$ and that the connection $\nabla_{0}$ restricted to
$(-2,2] \times N$ is invariant under the translation.

Let $M_{(-2,a]}$ denote $(M\setminus (-2,2] \times N) \cup ((-2,a] \times
N)$.
For instance the original $M$ can be denoted by $M_{(-2,2]}$.
Then for any $\ep>0$, there exists $\delta>0$ such that
if $\norm{R^{E_{0}}}<\delta$,
there exists a vector bundle $(E,\nabla)$ over $M_{(-2,6]}$ 
satisfying the following;
\\ (i) $\norm{R^{E}}<\ep$.
\\ (ii) The restriction of $(E,\nabla)$ to $M_{(-2,2]}$ is
isomorphic to $(E_{0},\nabla_{0})$.
\\ (iii) $(E,\nabla)$ is trivial and flat on $(4,6] \times N$.
\end{lem}

\begin{proof}
Choose $\ep_{0}>0$ sufficiently small. For $\{0\} \times N$ and $\ep_{0}$,
we can find $\delta = \delta(\ep_{0}) >0$ as in the preceding lemma
\ref{lem:simply connected}.
Suppose that $\norm{R^{E_{0}}}<\delta$. Then 
we obtain a global orthonormal frame ${e}$ for $E_{0}|_{\{0\} \times N}$
such that 
the connection 1-form $\omega_{0}$ with respect to $e$ satisfies
$\norm{\omega_{0}} < \ep_{0}$.
Let $E$ be a trivial Hermitian vector bundle without a connection on
$M_{(-2,6]}$ which is an extension of $E_{0}$.
Extending $e$ we obtain a orthonormal frame for $E$ denoted also by $e$.
Now compose a connection 1-form $\omega$
with respect to $e$ on on $(-2,6]\times N$ by
\begin{eqnarray}
\omega|_{(t,y)} = \chi(t)\omega_{0}|_{y}
\end{eqnarray}
where $\chi$ is a smooth function on $(-2,6]$ satisfying
\begin{eqnarray}
&&\chi(t)\left\{
\begin{array}{ll}
 \equiv 1, &\quad t<2 \\
 \equiv 0, &\quad t>4
\end{array}
\right.\\
&&0 \leq \frac{\dif \chi}{\dif t} \leq 1
\end{eqnarray}
Since $\omega|_{(t,y)} = \omega_{0}|_{y}$ on $(-2,2) \times N$,
the new connection denoted by $\nabla$ can be patched with $\nabla_{0}$.
\begin{eqnarray}
\norm{R^{E,\nabla}} &=& \norm{\omega \wedge \omega + \dif \omega}
\nonumber \\
&=& \norm{\chi(t)^{2} \omega_{0} \wedge \omega_{0} 
+ \dif \chi(t) \wedge \omega_{0} + \chi(t) \wedge \dif \omega_{0}}
\nonumber \\
&\leq & |\chi(t)|\norm{ \omega_{0} \wedge \omega_{0}
+ \dif \omega_{0}}
+ |\chi(t)^{2}-\chi(t)|\norm{ \omega_{0} \wedge \omega_{0} }
+ \norm{\dif \chi \wedge \omega_{0}}
\nonumber \\
&\leq & \norm{R^{E_{0},\nabla_{0}}} + \norm{\omega_{0}}^{2} 
+ \norm{\omega_{0}}
\nonumber \\ &\leq & \delta + \ep_{0} ^{2} + \ep_{0}
\end{eqnarray}
Hence taking $\ep_{0}$ depending on $\ep$ and $\delta$ depending
on $\ep_{0}$ sufficiently small, we obtain $\norm{R^{E,\nabla}}<\ep$.
Moreover $\omega|_{(t,y)}=0$ for $t>4$ means that $(E,\nabla)$ is 
flat on $(4,6] \times N$.
\end{proof}

\begin{dfn}
Let $M$ be a Riemannian manifold, and $n=p+q= \dim (M)$. Fix an
inclusion $\varphi \colon S^{p} \times D^{q} \hookrightarrow M$.
Define another (smooth) manifold $M\oshp$ as follows;
\begin{eqnarray}
M \oshp := (M \setminus \varphi(S^{p} \times D^{q}))
\cup_{\pd (\varphi(S^{p} \times D^{q}))} (D^{p+1} \times S^{q-1})
\end{eqnarray}
Remark that $\pd(S^{p} \times D^{q}) \cong S^{p} \times S^{q-1}
\cong \pd(D^{p+1} \times S^{q-1})$.
$M\oshp$ is called a manifold obtained by $p$-surgery, or 
surgery in codimension $q$, along
$\varphi \colon S^{p} \times D^{q} \hookrightarrow M$.
\end{dfn}
We assume that $M\oshp$ is equipped with a Riemannian metric
which coincides with the original one
outside a compact neighborhood of $(D^{p+1} \times S^{q-1})
\subset M\oshp$.

\begin{proof}[proof of theorem \ref{thm:surgery}]
Since the finiteness of $\Karea$ is invariant under deformations of 
Riemannian metrics on compact subsets,
we may assume that the "connecting region", the neighborhood of 
$\pd (D^{p+1} \times S^{q-1}) \subset M\oshp$ 
is isometric to $ S^{p} \times (-4,4) \times S^{q-1} $ equipped with a
canonical Riemannan metric.

Let $E_{0} = (E_{0},\nabla_{0})$ be a Hermitian vector bundle
equipped with a compatible connection.
It is sufficient to verify that for sufficiently
small $\delta>0$, $\norm{R^{E_{0}}}<\delta$ 
implies that all Chern numbers of $E_{0}$ are zero .

Let $\map{f}{M\oshp}{M\oshp}$ be a smooth Lipschitz map such that
$ f=\id $ outside $ S^{p} \times (-4,4) \times S^{q-1} $,
$f(x,t,y) = (x,0,y)$ for $|t|<2$, and $\norm{f_{\ast}}<2$. 
Consider $f^{\ast}E_{0}$, the pull-back of $E_{0}$ by $f$ 
equipped with the induced connection $f^{\ast} \nabla_{0}$.
Then $\norm{R^{f^{\ast}E_{0}}} \leq 2\delta$
and the connection is invariant under the translation
near the cylindrical boundary.
Since $\deg(f)=1$ the Chern numbers of $f^{\ast}E_{0}$ are
equal to those of $E_{0}$.

Cut $M\oshp$ along $S^{p} \times \{0\} \times S^{q-1}$ and 
remove $D^{p+1} \times S^{q-1}$ component.
Let the resulting manifold be denoted by $M'$.

In the case of $p\neq 1$, we can apply to $f^{\ast}E_{0}|_{M'}$
the preceding lemma \ref{lem:simply connected boundary}
to obtain a vector bundle
$E=(E,\nabla)$ over $M'_{(-2,6]}$ with $\norm{R^{E}}<\ep$
which is trivial and flat on $S^{p} \times (4,6] \times S^{q-1}$.
Remark that $\pd M' = S^{p} \times \{0\} \times S^{q-1}$ is 
simply connected by the condition $q\neq 2$.

Even in the case of $p=1$, we claim that there exist a such extension
of the vector bundle.
In fact,
consider the two copies of removed region $D^{2} \times S^{n-2}$
and the vector bundle
$f^{\ast}E_{0} \rightarrow (D^{2} \times S^{n-2})$ and reverse
the orientation of one of them.
We can patch them together along the boundary
for the invariance of the connection of $f^{\ast}E_{0}$ under the
translation near the cylindrical boundary.
Since the resulting manifold, the double of $D^{2} \times S^{n-2}$,
is homeomorphic to $S^{2}\times S^{n-2}$, by lemma \ref{lem:simply
connected}
there exists a global orthonormal frame $e$ for the resulting vector
bundle over $S^{2}\times S^{n-2}$ such that the connection 1-form
$\omega_{0}$ with
respect to $e$ satisfies $\norm{\omega_{0}} < \ep_{0}$.
Hence there exists a such orthonormal frame
for the restriction of $f^{\ast}E_{0}$ onto a neighborhood of $\pd M'$.
Then we can construct a vector bundle
$E=(E,\nabla)$ over $M'_{(-2,6]}$ with $\norm{R^{E}}<\ep$
which is trivial and flat on $S^{p} \times (4,6] \times S^{q-1}$
in the same way as the proof of lemma \ref{lem:simply connected boundary}
.
\\

In the following argument the condition $q \neq 2$ is not needed.
Deform the metric of $S^{p}\times D^{q}$ to have a product metric
near the boundary $S^{p}\times (-1,1) \times S^{q-1}$
so that it can be patched with $M'_{(-2,6]}$. The resulting manifold
is homeomorphic to $M$. Since $(E,\nabla)$ is trivial and flat on 
$S^{p}\times (4,6] \times S^{q-1} \subset
(M'_{(-2,6]}\cup S^{p}\times D^{q})$, it can be extended on
$S^{p}\times D^{q}$ trivially.

Let $X$ be $M\oshp \setminus M'$ and let $Y$ be $M \setminus M'$.
They are homeomorphic to $D^{p+1} \times S^{q-1}$ and $S^{p}\times D^{q}$
respectively.
Glue $X$ and $(-Y)$ together to compose a Riemannian manifold
homeomorphic to $S^{n}$ where $(-Y)$ is the orientation reversed $Y$.
Remark that $(f^{\ast}E_{0},f^{\ast}\nabla_{0})$ on $X$ and 
$(E,\nabla)$ on $(-Y)$ can be joined smoothly.
Hence they define a Hermitian vector bundle equipped with a compatible
connection $(E,\nabla)$ with a small curvature
$\norm{R} < \ep$ on $X \cup (-Y)$.

Since $\Karea (M) = \Karea (M' \cup Y) < \infty$ and $\Karea (X \cup (-Y))
< \infty$, 
there exist $\ep >0$ such that for any polynomial $p$, 

\begin{eqnarray}
&&\int _{M' \cup Y}p(c_{1}(E),c_{2}(E), \cdots ) = 0 
\nonumber \\
&&\int _{X\cup (-Y)}p(c_{1}(E),c_{2}(E), \cdots ) = 0 
\end{eqnarray}
Therefore,
\begin{eqnarray}
&&\int _{M\oshp}p(c_{1}(E_{0}),c_{2}(E_{0}), \cdots )
\nonumber \\
&=& \int _{M' \cup X}p(c_{1}(f^{\ast}E_{0}),c_{2}(f^{\ast}E_{0}), \cdots )
\nonumber \\
&=& \int _{M' \cup Y}p(c_{1}(E),c_{2}(E), \cdots )
+ \int _{X \cup (-Y)}p(c_{1}(E),c_{2}(E), \cdots )
\nonumber \\
&=&0
\end{eqnarray}
which implies $\Karea (M\oshp) < \frac{1}{\delta} < \infty$.
\end{proof}

\begin{rem}
Notice that surgery is an invertible operator.
Let $M \oshp$ be a Riemannian manifold obtained from $M$ by $p$-surgery.
Then $M$ is obtained by performing $(q-1)$-surgery to $M\oshp$.
$M= (M \oshp \setminus D^{p+1} \times S^{q-1}) \cup (S^{p} \times
D^{q})$.
So if both $p\neq 1$ and $q-1 \neq 1$ are satisfied, then
$\Karea (M\oshp)=\infty$ iff $\Karea (M)=\infty$.
\end{rem}
\begin{proof}[proof of corollary \ref{cor:connected sum}]
By lemma \ref{lem:conncetd sum trivial},
$\Karea (M_{1} \sharp M_{2})<\infty$
implies $\Karea (M_{1})<\infty$ and $\Karea (M_{2})<\infty$.
Suppose that both $\Karea (M_{1})$ and $\Karea (M_{2})$ are finite.
Remark that the K-area of the disjoint union
$\Karea (M_{1} \sqcup M_{2})$ is equal to
$\max \{ \Karea (M_{1}),\Karea (M_{2}) \}$.
Then we can apply the case of $p=0$ of the preceding theorem
\ref{thm:surgery} to conclude $\Karea (M_{1} \sharp M_{2}) < \infty$.
\end{proof}

Notice that we did not assume that $M$ is compact in theorem
\ref{thm:surgery} and so we can "localize" K-area in the following sense.

\begin{exa}\label{exa:cylinder}
Let $M_{\infty}$ is an oriented even dimensional Riemannian manifold
with a cylindrical end $(0,\infty) \times S^{n-1}$ and suppose that 
$M_{0}:= M_{\infty} \setminus ((0,\infty) \times S^{n-1})$ is compact.
Let $M$ be a compact manifold obtained by sewing a disk $D^{n}$ on
$M_{0}$.
$\Karea(M_{\infty}) = \infty$ if and only if $\Karea(M) = \infty$.
\end{exa}
\begin{proof}
This is a direct consequence of corollary \ref{cor:connected sum}.
In fact $M_{\infty}$ can be written as $M \sharp M'$
where $M'$ is a complete Riemannian manifold homeomorphic to
$\R ^{n}$ whose metric is a smoothing of the
$n$ dimensional hemispherical metric attached along
the canonical cylindrical metric on $[0,\infty) \times S^{n-1}$
of the same radius.
Since $\inf \Sc_{M'} >0$ and $M'$ is spin, $\Karea (M')$ is finite.
Therefore
$\Karea(M_{\infty}) = \infty$ if and only if $\Karea(M) = \infty$.
\end{proof}

Example \ref{exa:cylinder} suggests that the cylindrical region
$(0,\infty) \times S^{n-1}$ have no effect on finiteness or infiniteness
of $\Karea$.



Address

Department of Mathematics, Faculty of Science, Kyoto University

Sakyo-ku, Kyoto 606-8502, JAPAN

fukumoto@math.kyoto-u.ac.jp
\end{document}